\documentclass[a4paper]{article}
\usepackage{amsmath}
\usepackage{amsthm}
\usepackage{amssymb}
\usepackage{amsfonts}
\usepackage{graphicx}
\usepackage{mathrsfs}%for \mathscr
\usepackage{titlesec}

\usepackage[utf8]{inputenc}
\usepackage{hyperref}
\usepackage{url}

\usepackage[margin=1in]{geometry}
\usepackage{tikz-cd}
\theoremstyle{definition}
\newtheorem{question}{Question}[section]

\newtheorem{definition}[question]{Definition}
\newtheorem{remark}[question]{Remark}
\newtheorem*{remark*}{Remarque}
\newtheorem{example}[question]{Example}

\theoremstyle{plain}
\newtheorem{theorem}[question]{Theorem}
\newtheorem*{theorem*}{Theorem}
\newtheorem*{reponse*}{Answer}

\newtheorem{lemma}[question]{Lemma}
\newtheorem*{lemma*}{Lemma}

\newtheorem*{proposition*}{Proposition}

\title{\Large Injectivity radius of manifolds with a Lie structure at infinity}
\date{}
\author{Quang-Tu Bui}

\usepackage[T1]{fontenc}
\begin{document}
\maketitle
\titleformat*{\section}{\normalsize\bfseries\filcenter}

\begin{abstract}
Using Lie groupoids, we prove that the injectivity radius of a manifold with a Lie structure at infinity is positive. 
\end{abstract}

%\tableofcontents
\section*{Introduction}
Manifolds with a Lie structure at infinity were introduced by Ammann, Lauter and Nistor in \cite{aln}, forming a class of non-compact complete Riemannian manifolds of infinite volume. In the same article, they conjectured that the injectivity radius of a (connected) manifold with Lie structure at infinity is positive. In this paper, we give a proof of this conjecture using the associated groupoid given by \cite{cf} and \cite{de}. Together with the results from \cite{aln}, this implies that manifolds with a Lie structure at infinity are of bounded geometry. In particular, the hypothesis of injectivity radius in \cite{aln2} is now automatically satisfied, as well as in \cite{ma}, where positivity of the injectivity radius is used to obtain uniform parabolic Schauder estimates. Bounded geometry also yields uniform elliptic Schauder estimates, see \cite{con} for a recent application in this direction.\par
\textbf{Acknowledgement.} \textit{The author thanks his PhD advisor Frédéric Rochon for suggesting the problem and the approach, and also wishes to thank Bernd Ammann, Claire Debord and Victor Nistor for helpful discussions.}
\section{Preliminaries}
Following \cite{aln} and \cite{nwx}, we recall some definitions and facts. 
\begin{definition}
A \textit{groupoid} is a small category $G$ in which every morphism is invertible.
\end{definition}
The objects of the category are also called \textit{units}, and the set of units is denoted by $G^0$. The set of morphisms is denoted by $G^1$. The range and domain maps are denoted respectively $r,d:G^1\to G^0$. The multiplication operator $\mu$ is defined on the set of composable pairs of morphisms by:
$$\mu:G^2=G^1\times_{G^0}G^1=\{(g,h):d(g)=r(h)\}\to G^1$$
The inversion operation is a bijection $\iota:g\mapsto g^{-1}$ of $G^1$. The identity morphisms give an inclusion $u:x\mapsto\mathrm{id}_x$ of $G^0$ into $G^1$.\par
\begin{definition}
An \textit{almost differentiable groupoid} $G=(G^0,G^1,d,r,\mu,u,\iota)$ is a groupoid such that $G^0$ and $G^1$ are manifolds with corners (\cite{me}), the structural
maps $d, r,\mu, u, \iota$ are differentiable, and the domain map $d$ is a submersion.
\end{definition}
Consequently, for an almost differentiable groupoid, $\iota$ is a diffeomorphism, $r=d\circ\iota$ is a submersion and each fiber $G_x=d^{-1}(x)\subset G^1$ is a smooth manifold whose dimension $n$ is constant on each connected component of $G^0$.\par
Following the convention in \cite[p. 578]{cf}, we require $G^0$ and $d^{-1}(x)$ to be Hausdorff (for all $x\in G^0$), but not necessarily $G^1$ to avoid excluding important cases.\par
From now on, \textit{Lie groupoid} will stand for almost differentiable groupoid, \textit{manifold} will stand for \textit{manifold with corners} and \textit{smooth manifold} will stand for manifold wihout corners.\par
A Lie groupoid is called $d$-\textit{simply connected} if its $d$-fibers $G_x=d^{-1}(x)$ are simply connected (\cite{cf}).\par
\begin{definition}
A \textit{Lie algebroid} $A$ over a manifold $M$ is a vector bundle $A$ over $M$, together with a Lie algebra structure on the space $\Gamma(A)$ of smooth sections of $A$ and a bundle map $\rho : A \to TM$, extended to a map $\rho_{\Gamma}: \Gamma(A) \to\Gamma(TM)$ between sections of these bundles, such that
\begin{enumerate}
\item $\rho_\Gamma([X,Y])=[\rho_\Gamma(X),\rho_\Gamma(Y)]$
\item $[X, fY ] = f[X,Y] + (\rho_\Gamma(X)f)Y$ 
\end{enumerate}
for any smooth sections $X$ and $Y$ of $A$ and any smooth function $f$ on $M$.
\end{definition}
There is a Lie algebroid $A(G)$ associated to a Lie groupoid $G$, constructed as follows: let $T_{\mathrm{vert}}G=\mathrm{ker}d_*=\cup_{x\in G^1}TG_x\subset TG^1$ be the vertical bundle over $G^1$. Then $A(G)=T_{\mathrm{vert}}G|_{G^0}$ is the structural bundle of the Lie algebroid over $G^0$. The anchor map is given by
$$r_*|_A:A\to TG^0$$
(\cite{aln2}). The Lie bracket of $\Gamma(A)$ is the Lie bracket of $\Gamma(T_{\mathrm{vert}}G)$ restricted to right invariant sections.\par
\begin{definition}
A Lie algebroid $A$ over a manifold $M$ is said to be \textit{integrable} if there exists a Lie groupoid $G$ such that $G^0=M$ and $A$ is isomorphic to the Lie algebroid associated to $G$. $G$ is said to \textit{integrate} $A$.
\end{definition}
\begin{remark}
\label{cf}
There might be more than one Lie groupoid integrating a Lie algebroid. However, by \cite[Lie I]{cf}, if a Lie algebroid over a smooth manifold is integrable, there is a unique $d$-simply connected Lie groupoid integrating it.
\end{remark}
\begin{example}
\begin{enumerate}
\item Any Lie group is a Lie groupoid with the set of units being a singleton.
\item (\cite[Example 4, Section 4]{nwx}) Let $M$ be a smooth manifold. Let $\widetilde{M}$ be the universal covering of $M$. Let $H=(\widetilde{M}\times\widetilde{M})/{\pi_1(M)}$. Then $H$ is naturally a $d$-simply connected Lie groupoid with the set of units being $M$, and the associated Lie algebroid being $\mathrm{id}:TM\to TM$. It is called the homotopy groupoid.
\item The space of continuous paths on a topological space modulo homotopy equivalence forms a groupoid which is called the fundamental groupoid.
\end{enumerate}
\end{example}
We recall the definitions and basic properties of manifolds with Lie structures at infinity. For details and proofs, we refer to \cite{aln}.\par
\begin{definition} A \textit{structural Lie algebra} of vector fields on a manifold $M$ (possibly with corners) is a subspace $\mathcal{V}\subset\Gamma(TM)$ of the real vector space of vector fields on $M$ with the following properties:
\begin{enumerate}
\item $\mathcal{V}$ is closed under Lie brackets;
\item $\mathcal{V}$ is a finitely generated projective $\Gamma(M)$-module;
\item The vector fields in $\mathcal{V}$ are tangent to all faces in $M$.
\end{enumerate}
\end{definition}
Denote by $\mathcal{V}_b(M)\subset\Gamma(TM)$ the subspace of vector fields tangent to all faces in $M$. This is a structural Lie algebra of vector fields, and any structural Lie algebra is a subspace of $\mathcal{V}_b(M)$ (\cite[Example 2.5]{aln}).\par
By the Serre-Swan theorem, given a structural Lie algebra of vector fields $\mathcal{V}$ on $M$, there exists a vector bundle $A=A_\mathcal{V}\to M$ such that $\mathcal{V}\simeq \Gamma(A_\mathcal{V})$, and there exists a natural vector bundle map $\rho:A_\mathcal{V}\to TM$ such that the induced map $\rho_\Gamma:\Gamma(A_\mathcal{V})\to\Gamma(TM)$ is identified with the inclusion map $\mathcal{V}\subset\Gamma(TM)$. The vector bundle $A_\mathcal{V}$ is then a Lie algebroid with anchor map $\rho$.
\begin{definition} A \textit{Lie structure at infinity} on a smooth manifold $M_0$ is a pair $(M,\mathcal{V})$, where
\begin{enumerate}
\item $M$ is a compact manifold, possibly with corners, and $M_0$ is the interior of M;
\item $\mathcal{V}$ is a structural Lie algebra of vector fields on $M$;
\item $\rho:A_{\mathcal{V}}\to TM$ induces an isomorphism on $M_0$, that is, $\rho|_{M_0}: A|_{M_0}\to TM_0$ is an isomorphism of vector bundles.
\end{enumerate}
\end{definition}
\begin{definition} A \textit{Riemannian manifold with a Lie structure at infinity} is a manifold with a Lie structure at infinity $(M,\mathcal{V})$ endowed with a bundle metric $g$ on $A=A_\mathcal{V}$. In particular, $g$ defines a Riemannian metric on $M_0$ via the anchor map.
\end{definition}
A Riemannian manifold with a Lie structure at infinity has infinite volume (\cite[Proposition 4.1]{aln}), bounded curvature (\cite[Corollary 4.3]{aln}) and is complete (\cite[Corollary 4.9]{aln}). Sufficient conditions for the positivity of the injectivity radius are given in \cite[Theorem 4.14]{aln} and \cite[Theorem 4.17]{aln}.
\section{Injectivity radius of a manifold with Lie structure at infinity}
The following theorem is due to Debord (\cite[Theorem 2]{de}, see also \cite[Corollary 5.9]{cf}).
\begin{theorem}[Debord]
\label{dcf}
Every almost injective Lie algebroid over a smooth manifold is integrable.
\end{theorem}
This has the following implication for Lie structures at infinity.
\begin{theorem}
\label{corners}
Any Lie algebroid over a manifold with corners associated with a Lie structure at infinity is integrable.
\end{theorem}
\begin{proof}
This extension of Theorem \ref{dcf} to manifolds with corners is well-known to experts. However, since no explicit proof seems to be available in the literature, we will provide one for the convenience of the readers.\par
Let $(M,\mathcal{V})$ be a Lie structure at infinity of $M_0$ and $A=A_{\mathcal{V}}$ be the corresponding structural vector bundle. Taking two copies of $M$ and gluing them along a maximal subset of disjoint boundary hypersurfaces, we obtain a compact manifold with corners $M_1$ with at least one hypersurface less. Repeating this operation finitely many times, we obtain a closed manifold $\widetilde{M}$ with a finite group $\Gamma$ acting on $\widetilde{M}$ such that $\widetilde{M}/\Gamma\simeq M$ topologically. Now, by \cite[Exercise 1.6.2]{me}, $M$ is naturally an orbifold. Changing the smooth structure on $\widetilde{M}$, one can in fact ensure that the quotient map $q:\widetilde{M}\to M$ is such that $q^*(\mathcal{C}^\infty(M))=\mathcal{C}^\infty(\widetilde{M})_{\Gamma}=\{f\in\mathcal{C}^\infty(\widetilde{M}):\forall g\in\Gamma, f\circ g=f\}$. In a suitable local chart, $q$ can be written as $(x_1,\dots,x_k,x_{k+1},\dots,x_n)\mapsto(x_1^2,\dots,x_k^2,x_{k+1},\dots,x_n)$ where $k$ is the depth of the point $(0,0,\dots,0)$. Let $\widetilde{\mathcal{V}}=\mathcal{C}^\infty(\widetilde{M})\otimes_{\mathcal{C}^\infty(\widetilde{M})_{\Gamma}}q^*\mathcal{V}\subset\mathfrak{X}(T\widetilde{M})$ be the pull-back of the structural vector fields. For instance, if $\mathcal{V}_b(M)$ is the space of vector fields tangent to the faces of $M$, then $\widetilde{\mathcal{V}}_b(M)$ is the space of vector fields on $\widetilde{M}$ which are tangent to $q^{-1}(\partial M)$ (the union of some closed submanifolds of $\widetilde{M}$).\par
Now, $\widetilde{\mathcal{V}}$ is a finitely generated projective $\mathcal{C}^\infty(\widetilde{M})$-module. To see this, it suffices to show that $\widetilde{\mathcal{V}}$ is locally free of rank $k$ for some $k$. Given $p\in\widetilde{M}$, then since $\mathcal{V}$ is locally free of rank $k$ for some $k$, there exist $v_1,\dots,v_k\in\mathcal{V}$ which locally and freely span $\mathcal{V}$ near $q(p)$. This means $\widetilde{\mathcal{V}}$ is locally and freely spanned by $q^*v_1,\dots,q^*v_n\in\widetilde{\mathcal{V}}$ near $p$, showing that $\widetilde{\mathcal{V}}$ is locally free of rank $k$ as claimed.\par
By the Serre-Swan theorem, we have a vector bundle $A_{\widetilde{V}}$ over $\widetilde{M}$ with the space of smooth sections $\mathcal{C}^\infty(\widetilde{M},A_{\widetilde{V}})=\widetilde{\mathcal{V}}$. Clearly the inclusions $\widetilde{\mathcal{V}}\subset\widetilde{\mathcal{V}_b}\subset\mathcal{C}^\infty(\widetilde{M},T\widetilde{M})$ induce an anchor map, so that $A_{\widetilde{V}}$ is naturally an almost injective Lie algebroid. By the Theorem \ref{dcf} and the Remark \ref{cf}, there exists therefore a $d$-simply connected groupoid $\widetilde{G}$ integrating $A_{\widetilde{V}}$. Each element $g\in\Gamma$ induces an automorphism $\rho(g):A_{\widetilde{V}}\to A_{\widetilde{V}}$, and by Lie II, an automorphism on $\widetilde{G}$. Hence we have an action of the group $\Gamma$ over $\widetilde{G}$. The quotient $\widetilde{G}/\Gamma$ is then the desired $d$-simply connected Lie groupoid integrating $(M,\mathcal{V})$.
\end{proof}
Let $M_0$ be a connected smooth manifold with a Lie structure at infinity $(M,\mathcal{V})$. By Theorem \ref{corners}, there exists a $d$-simply connected groupoid $G=(M,G^1,d,r,\mu,u,\iota)$ with units $M$ such that $A(G)\simeq A$ as Lie algebroids over $M$. Therefore $A(G)$ is equipped with an inner product also noted $g$. The anchor map is given by $r_*:A(G)\to TM$.\par
We have an isomorphism $r^*A(G)\simeq T_{\mathrm{vert}}G$ where $r^*A(G)$ is the pull-back of $A(G)$ via the range map $r:G\to M$ (\cite[(19)]{aln2}). Explicitly, for $p\in G$, $(r^*A(G))_p=A(G)_{r(p)}=T_{r(p)}G_{r(p)}$. The vector bundle $r^*A(G)$ is equipped with a metric induced by the metric $g$ on $A(G)$, hence so is $T_{\mathrm{vert}}G$. Therefore each $G_x$ becomes a Riemannian manifold for all $x\in M$.\par 
Let $G_x^x=\{g\in G_x:r(g)=x\}$. For $x\in M_0$, $G_x^x$ is a discrete group since $T_xG_x^x$ is of dimension $0$ (being the kernel of the map $r_*:A(G)_x\to T_xM_0$).\par
\begin{lemma}(\cite[page 733]{aln2})
\label{cov}
If $A\to TM$ is the Lie algebroid associated with a Lie structure at infinity and $G$ is the corresponding $d$-simply connected Lie groupoid, then for all $x\in M_0, r:G_x\to M_0$ is a covering map with group $G_x^x$.
\end{lemma}
\begin{proof}
By \cite[Proposition 1.1]{cf}, for all $x\in M_0$, $r(G_x)\subset M_0$ (which is the leaf of the singular foliation of $A$ passing by $x$). On the other hand, $G|_{M_0}$ is the unique $d$-simply connected Lie groupoid which integrates $TM_0$, and therefore it isomorphic to the homotopy groupoid $(\widetilde{M_0}\times\widetilde{M_0})/{\pi_1(M_0)}$. Consequently, $M_0=r(G_x)$ for all $x\in M_0$.\par
Now, by definition of a Lie structure at infinity, $r_*:T_yG_x\to T_{r(y)}M_0$ is an isomorphism. This means that $r:G_x\to M_0$ is a local diffeomorphism. Moreover, $g_1,g_2\in G_x$ with $r(g_1)=r(g_2)$ if and only if there exists $h=g_1^{-1}g_2\in G_x^x$ such that $g_2=g_1h$. That is, $r:G_x\to M_0$ is a covering map with group $G_x^x$.
\end{proof}
\begin{theorem}
Let $M_0$ be a connected smooth manifold with a Lie structure at infinity $(M,\mathcal{V})$. Then for any Riemannian metric $g$ on $A$, the injectivity radius of $(M_0,g)$ is positive.
\end{theorem}
\begin{proof}
We prove the theorem by contradiction. Suppose that the injectivity radius of $(M_0,g)$ is zero, then, as the curvature is bounded, there is a sequence of geodesic loops $c_i:[0,a_i]\to M_0$, parametrized by arc-length, with $a_i\to 0$. By compactness of $M$, we can suppose that $c_i(0)$ converges to a point $p\in M$. We have $p\in\partial M$ since the injectivity radius is positive in any compact region of $M_0$.\par
Let $U$ be a local chart of $M$ containing $p$ such that $U$ is contractible. 
\begin{lemma}
There exists a number $N>0$ such that $\forall n>N$, the loop $c_n$ is contained in $U$.
\end{lemma}
\begin{proof}
Let $(x_1,\dots,x_k,y_1,\dots,y_l)$ be a set of local coordinates centered at the point $p$ with $x_i\geq 0$ for all $i$ and $p=(0,\dots,0)$. Let $g_b=
\sum_{i=1}^k\frac{dx_i^2}{x_i^2}+\sum_{i=1}^ldy_i^2$ be a local $b$-metric and $g_0=\sum_{i=1}^kdx_i^2+\sum_{i=1}^ldy_i^2$ be a local metric with boundary. Since the structural vector fields are tangential vector fields ($\mathcal{V}\subset\mathcal{V}_b$), taking $U$ smaller if needed, there exist constants $C,K>0$ such that $g\geq Cg_b\geq CKg_0$ in $U\cap M_0$. Let $l^t(c_i),l^t_b(c_i),l^t_0(c_i)$ denote the lengths of the segment $[c_i(0),c_i(t)]$ (of the geodesic loop $c_i$) with respect to the metric $g$, the local $b$-metric $g_b$ and the local metric with boundary $g_0$ respectively (suppose that the segment is contained in $U$). Let $\varepsilon>0$ be such that $B_0(p,\varepsilon)=\{x\in\mathbb{R}_+^k\times\mathbb{R}^l:d_0(x,p)<\varepsilon\}\subset U$ (where $d_0$ is the distance with respect to the metric $g_0$, well-defined on $B_0(p,\varepsilon)$). Since $a_i\to 0$, there exists $N_1$ such that $a_i<\min(\frac{\varepsilon}{4},CK\frac{\varepsilon}{4})$ for all $i>N_1$. Since $c_i(0)\to p$, there exists $N_2$ such that $d_0(p,c_i(0))<\frac{\varepsilon}{4}$ for all $i>N_2$. Let $N=\max(N_1,N_2)$.\par
Now let $n$ be any number greater than $N$. Suppose that the loop $c_n$ is not contained in $U$. Then it is not contained in $B_0(p,\frac{\varepsilon}{2})$. Thus there exists $t\in[0,a_n]$ minimal such that $d_0(c_n(t),p)=\frac{\varepsilon}{2}$. Then we have $d_0(c_n(0),c_n(t))\geq\vert d_0(c_n(t),p)-d_0(c_n(0),p)\vert\geq\frac{\varepsilon}{4}$, which implies $a_i=l(c_i)\geq l^t(c_i)\geq CKl^t_0(c_i)\geq CKd_0(c_n(0),c_n(t))\geq CK\frac{\varepsilon}{4}$, which is a contradiction. Therefore the loop $c_n$ is contained in $U$.\par
The lemma is proven.
\end{proof}
Hence, without loss of generality, we can suppose that the loops are contained in $U$.\par
Denote by $G=(M,G^1,d,r,\mu,u,\iota)$ the $d$-simply connected groupoid integrating $A_\mathcal{V}\to TM$. Since $U$ is contractible, the fundamental class of each loop $c_i$ is trivial, therefore by Lemma \ref{cov} we can lift $c_i$ to a geodesic loop $\widetilde{c}_i$ in $G_{c_i(0)}$ (i.e. $\widetilde{c}_i:[0,a_i]\to r^{-1}(U)\cap G_{c_i(0)}$) such that the base points are $\widetilde{c}_i(0)=\widetilde{c}_i(a_i)=c_i(0)=c_i(a_i)$.\par
Let $S(T_{\mathrm{vert}}G)=\{x\in T_{\mathrm{vert}}G:\Vert x\Vert=1\}$. We have a natural projection $\pi:S(T_{\mathrm{vert}}G)\to G^1$. On $S(T_{\mathrm{vert}}G)$ we have a flow $\Psi$ which, over each $d$-fiber $G_x$ of $d:G^1\to G^0$, corresponds to the geodesic flow of $G_x$. The geodesic loops on $G_x$ correspond to segments $[P_i,Q_i]$ of the flow $\Psi$ on $S(TG_x)$ (with $Q_i=\Psi_{a_i}(P_i)$). We have two sequences $P_i=(\widetilde{c}_i(0),\dot{\widetilde{c}}_i(0))$ and $Q_i=(\widetilde{c}_i(a_i),\dot{\widetilde{c}}_i(a_i))$ in $S(A)\subset S(T_{\mathrm{vert}}G)$. By compactness of $S(A)$ and $M$, there exists a subsequence such that $P_i\to P\in S(TG_p)$ and $Q_i\to Q\in S(TG_p)$.\par
Since $a_i\to 0$, we have $P=Q$.  In a local chart, we can write $(\frac{Q_i-P_i}{a_i},c_i(0))\to (w,p)$. Since $a_i\to 0$, $w=\dot{\Psi}(P)$. Since $P_i,Q_i\in (S(A))_{c_i(0)}$ for all $i$, $w$ is tangent to the fiber $S(A)_p=S(TG_p)$, which is a contradiction (for $\Psi$ is the geodesic flow over $G_p$).\par
\end{proof}
\begin{remark}
In \cite{aln}, a flow $\Phi$ is defined on $S(A)$ extending the geodesic flow on $S(TM_0)$. However, $\Phi$ itself is not quite a geodesic flow since typically it has fixed points at the boundary. Our approach does not seem to work with this flow. Indeed, to each geodesic loop $c_i:[0;a_i]\to M_0$, we have a corresponding segment $\Phi_i: [0; a_i] \to S(A)$. By considering a convergent subsequence, the limit of $(c_i(0),\dot{c}_i(0))$ is a point $v$ contained in $\partial S(A)=S(A)|_{\partial M}$. The limit of $c_i(0)$ is a point $p = \pi(v)$ in $\partial M$. In the notations of \cite{aln}, we have $(\pi^\# r_*)(H_v(v)) = 0$ and $r_*(v) = 0$. In particular, the flow $\Phi$ at $v$ is stationary: $\forall t, \Phi_t(v) = v$. This, however, is not sufficient to obtain a contradiction, since at the boundary, $\Phi$ may have some fixed points as mentioned above.\par
\end{remark}
\nocite{*}
\bibliographystyle{plain}
\bibliography{bib}
\end{document}